\documentclass[11pt,oneside,leqno]{amsart}
\usepackage{amsxtra}
\usepackage{amsopn}
\usepackage{amsmath,amsthm,amssymb}
\usepackage{amscd}
\usepackage{amsfonts}
\usepackage{latexsym}
\usepackage{verbatim}
\usepackage[matrix,arrow]{xy}
\usepackage{array}
\usepackage{multirow}
\usepackage{bigstrut}
\usepackage{rotating}
\usepackage{color,enumitem}

\newtheorem{theorem}{Theorem}
\newtheorem*{mtheorem}{Main Theorem}
\newtheorem{definition}{Definition}

\newtheorem{prop}{Proposition}

\newtheorem{remark}{Remark}

\begin{document}

\title[Formality of almost abelian solvmanifolds]{Formality and symplectic structures of almost abelian solvmanifolds}

\author{Maura Macr\`i} 
\address{M. Macr\`i: Dipartimento di Matematica G. Peano \\ Universit\`a di Torino\\
Via Carlo Alberto 10\\
10123 Torino\\ Italy.} \email{maura.macri@unito.it}

\begin{abstract}
In this paper we study some properties of almost abelian solvmanifolds using minimal models associated to a fibration. In particular we state a necessary and sufficient condition to formality and a method for finding symplectic strucures of this kind of solvmanifolds.
\end{abstract}

\maketitle

\section{Introduction}

Nilmanifolds and solvmanifolds are compact quotients of (respectively) nilpotent and solvable Lie groups  by a lattice.
They have been intensively studied from many points of view (geometry, topology, group theory) since they are, on the one hand, spaces for which the computation of some of their invariants is tractable and, on the other hand, they are involved enough to show all sort of different behaviors.

Typical examples in this context are the Nomizu Theorem for nilmanifolds, which states that their de Rham cohomology agrees with the one of the Lie algebra  \cite{nomizu}, and Benson and Gordon result on the non existence of any K\"ahler structure on a nilmanifold (unless it is a torus) \cite{BG}. 
K\"ahler manifolds are quite relevant within rational homotopy theory: in \cite{DGMS} it was shown that a compact K\"ahler manifold is formal. Hasegawa, using an explicit description of the minimal model of a nilmanifold,  proved that a non toral nilmanifold cannot be formal \cite{ha}, yielding an alternate proof of the above mentioned result by Benson and Gordon. In the case of solvmanifolds, Hasegawa proved in
\cite{Ha2} that a solvmanifold carries a K\"ahler metric if and only if it is covered by a finite quotient of a complex torus, which has the structure of a complex torus bundle over a complex torus. 

In this paper we study the formality and the symplectic structures of almost abelian solvmanifolds, i.e. compact homogeneous spaces $S=G/\Gamma$, where the solvable Lie group $G$ is a semidirect product $\mathbb{R}\ltimes\mathbb{R}^n$  and \mbox{$\Gamma=\mathbb{Z}\ltimes\mathbb{Z}^n.$}

Our idea starts from a work of Oprea and Tralle in which the theory of minimal models is used to compute the cohomology of some solvmanifolds \cite{OT2}. 
Indeed the above mentioned result of Nomizu does not apply in gene-\\ral to the cohomology of solvmanifolds.
Almost abelian solvmanifolds are probably the most tractable class of solvmanifolds whose cohomology does not in general agree with the one of the Lie algebra.

In particular, in \cite{OT2}, the authors use a theorem of Felix and Thomas (Theorem \ref{fibration}) in which is described the model of a fibration and they apply it to the \textit{Mostow fibration} 
\begin{equation}\label{mostow}
N / \Gamma_N = (N \Gamma) / \Gamma \hookrightarrow G / \Gamma\longrightarrow G / (N \Gamma) = \mathbb{T}^k,
\end{equation} 
where $N$ is the nilradical of $G$, associated to every solvmanifold.

\smallskip 

This construction is related to a submodule $U$ of the cohomology algebra of the abelian Lie algebra $\mathbb{R}^n \; H^*(\mathbb{R}^n)$ computed using (\ref{mostow}) and defined in \cite{OT, OT2} (cf. Section~\ref{preliminaries}).

Rather than using this theory to compute the cohomology of the solvmanifold (cf. \cite{noi}), we find some of its properties. Indeed, while in general the submodule $U$ is difficult to compute, its construction is quite simple for almost abelian solvmanifolds. Hence we are able to find some properties of $U$ and relate them to those of the solvmanifold.

Let $(\mathcal{M}_U,d)$ be a minimal commutative differential graded algebra (cdga), such that its cohomology algebra is isomorphic to $U$, then this algebra is a subspace of the minimal model $(\mathcal{M}_S,D)$ of the solvmanifold $S$, (Theorem \ref{fibration}). 

In particular in section \ref{formal} we find a necessary and sufficient condition for the formality of $S$: 
\begin{mtheorem}
If $\mathcal{M}_S$ is of finite type, then $S$ is formal if and only if $\ker D|_{\mathcal{M}_U}=\ker d$.
\end{mtheorem}

\medskip

In section \ref{symplectic} we give a method to find symplectic forms on an almost abelian solvmanifold and in the last section we give two examples of application of the Main Theorem.

\bigskip 

\section{Preliminaries}\label{preliminaries}

\begin{definition}
An almost abelian solvmanifold is a quotient $S=G/\Gamma$ where the solvable Lie group $G$ and its lattice $\Gamma$ are semidirect products of kind \\$G=\mathbb{R}\ltimes_{\varphi}\mathbb{R}^n, \quad \Gamma=\mathbb{Z}\ltimes_{\varphi|_{\mathbb{Z}}}\mathbb{Z}^n$.
\end{definition}

In particular if $\mathfrak{g}$ is the Lie algebra of $G$, then $\mathfrak{g}=\mathbb{R}\ltimes_{\mbox{ad}_{X_{n+1}}}\mathbb{R}^n$, where $\mathbb{R}=\langle X_{n+1} \rangle$ and $\mathbb{R}^n=\langle X_1, \cdots, X_n \rangle$, and $\varphi(t):= e^{t\mbox{ad}_{X_{n+1}}}$.

\medskip 

In this case the Mostow fibration is $$\mathbb{R}^n / \mathbb{Z}^n \hookrightarrow S
\longrightarrow \mathbb{R} / \mathbb{Z}$$

and we want to apply to this fibration the following theorem of Felix and Thomas

\begin{theorem} \label{fibration}\cite{OT, OT2}
Let $F \rightarrow E \rightarrow B$ be a fibration and let $U$ be the largest $\pi_1(B)$-submodule of $H^*(F,\mathbb{Q})$ on which $\pi_1(B)$ acts nilpotently. Suppose that $H^*(F,\mathbb{Q})$ is a vector space of finite type and that $B$ is a nilpotent space, then in the Sullivan model of the fibration 

\begin{displaymath}
\xymatrix{\mathcal{A}(B) \ar[r]  & \mathcal{A}(E) \ar[r] & \mathcal{A}(F) \\
(\bigwedge X, d_X) \ar[r] \ar[u]^\sigma & (\bigwedge (X\oplus Y), D)  \ar[u]^\tau \ar[r] & (\bigwedge Y, d_Y) \ar[u]^\rho }
\end{displaymath}
the cdga homomorphism $\rho: (\bigwedge Y, d_Y)  \to \mathcal{A}(F) $ induces an isomorphism $\rho^*: H^*(\bigwedge Y, d_Y)\rightarrow U$.
\end{theorem}

\bigskip 

We give here only basic definitions of cdga and minimal models and we refer to \cite{FOT} for a depth study of these topics.

\begin{definition}
Let $\mathbb{K}$ ba a field of characteristic $0$. A graded $\mathbb{K}$-vector space is a family of $\mathbb{K}$-vector spaces $\mathcal{A}=\{\mathcal{A}^{p}\}_{p\geq 0}$. An element of $\mathcal{A}$ has degree $p$ if it belongs to  $\mathcal{A}^{p}$. \\A commutative differential graded $\mathbb{K}$-algebra, cdga, $(\mathcal{A},d)$ is a graded $\mathbb{K}$-vector space $\mathcal{A}$ together with a multiplication $$\mathcal{A}^{p}\otimes\mathcal{A}^{q}\rightarrow\mathcal{A}^{p+q}$$ that is associative, with unit $1 \in \mathcal{A}^{0}$ and commutative in the graded sense, i.e. $\forall a \in \mathcal{A}^{p}, b \in \mathcal{A}^{q}\quad a\cdot b= (-1)^{pq}b\cdot a$, and with a differential $$d:\mathcal{A}^{p}\rightarrow\mathcal{A}^{p+1}$$ such that $d^2=0$ and $\;\forall\, a \in \mathcal{A}^{p}, b \in \mathcal{A}^{q}\quad d(a\cdot b)=da\cdot b +(-1)^{p}a\cdot db$.
\end{definition}

\medskip 

Given a $\mathbb{K}$-cdga $(\mathcal{A},d)$ its cohomology algebra $H^{*}(\mathcal{A},\mathbb{K})$ is well defined and it is a $\mathbb{K}$-cdga with $d\equiv 0$.

\medskip 

\begin{definition}\label{morp}
A cdga homomorphism $f:(\mathcal{A},d_{\mathcal{A}})\rightarrow(\mathcal{B},d_{\mathcal{B}})$ is a family of homomorphisms $f^p:\mathcal{A}^p\rightarrow\mathcal{B}^p$ such that $f^{p+q}(a\cdot b)=f^p(a)\cdot f^q(b)$ and  $d_{\mathcal{B}}f^p=f^pd_{\mathcal{A}}$.
\end{definition}

\medskip 

\begin{definition}
A cdga $(\mathcal{M},d)$ is minimal if it is free commutative,\\ i.e. $\mathcal{M}=\bigwedge V$ with $V$ graded vector space, and there exist a ordered basis $\{x_{\alpha}\}$ of $V$ such that $V^0=\mathbb{K},\;dV\subset \bigwedge^{\geq 2}V$ and $dx_{\alpha}\in \bigwedge(x_{\beta})_{\beta<\alpha}$, where with $\bigwedge^{\geq 2}V$ we mean $\bigwedge^{i}V$ with $i\geq 2$.
\\ A minimal model of the cdga $(\mathcal{A},d)$ is a minimal cdga $(\mathcal{M},d)$ together with a cdga quasi isomorphism $\psi:\mathcal{M}\rightarrow\mathcal{A}$, i.e. a morphism that induces an isomorphism on cohomology. 
\end{definition}

\medskip 

For every topological space $T$ Sullivan defined a $\mathbb{Q}$-cdga $\mathcal{A}(T)$ associated to $T$. We refer to \cite{FOT} for its definition, we only need to know that its cohomology is the cohomology of the space $T$ over the constant sheaf $\mathbb{Q}$, then we can apply Theorem \ref{fibration} to differential manifolds.

In particular by definition of Sullivan model of a fibration \cite{FOT}, we have that
\begin{itemize}
\item $(\bigwedge X, d_X)$ and $(\bigwedge Y, d_Y)$ are minimal cdga,
\item $\sigma$ and $\tau$ are quasi isomorphisms,
\item $\forall x \in X \; Dx=d_Xx$ and $\forall y \in Y \; Dy=d_Yy+cx\wedge y'$ with $c \in \mathbb{Q}, \, x\in \bigwedge X^+$ and $y' \in \bigwedge Y^{< y}$, where with $\bigwedge X^+$ we mean all the elements in $\bigwedge X$ with degree greater than $0$ and with $\bigwedge Y^{< y}$ the subalgebra of $\bigwedge Y$ generated by all the generators prior to $y$ with respect to an order among the basis of $Y$.
\end{itemize}

\medskip 

To apply this result to differential manifolds we consider the field $\mathbb{R}$ instead of $\mathbb{Q}$ and in the case of the Mostow fibration associated to an almost abelian solvmanifold we have that $X$ is generated only by a closed element $A$ of degree one, i.e. $(\bigwedge X, d_X)= (\bigwedge (A),0)$.

\smallskip 

In general, finding $U$ is very difficult, but when the solvmanifold is almost abelian the monodromy action of $\mathbb{Z} \cong \pi_1(\mathbb{R} / \mathbb{Z})$ on $H^*(\mathbb{R}^n / \mathbb{Z}^n)$ is exploited by the transpose of the twist action that defines the semidirect sum \mbox{$\mathfrak{g}=\mathbb{R}\ltimes \mathbb{R}^n$,} this means that the action is given by $$(\bigwedge \varphi^t)^*:\mathbb{Z} \rightarrow \mbox{Aut}(H^*(\mathbb{R}^n)), $$ where $(\bigwedge \varphi^t)^*$ is the restriction of $\varphi$ to $\mathbb{Z}$ induced on the exterior algebra of the dual (Chevalley-Eilenberg complex) and then on the cohomology (see \cite[Theorems 3.7 and 3.8]{OT}).

To simplify the notation we denote the action $(\bigwedge \varphi^t)^*$ with $\varphi$. \\By definition of nilpotent action we have that a form $\alpha$ is in $U$ if and only if there exist a constant $k \in \mathbb{N}^+$ such that $(\varphi-\mbox{Id})^k(\alpha)=0$, where Id is the identity map.

\medskip

The definition of $U$ can be further simplified by the following proposition

\begin{prop}\label{kasuya}
$\alpha \in U$ if and only if $\varphi_s(\alpha)=\alpha$, where $\varphi_s$ is the semisimple part of $\varphi$.
\end{prop}

\begin{proof}
We give the proof in 4 steps:
\begin{enumerate}
\item \textit{we can prove the proposition on the complexification:}\\
let $V$ a generic real vector space generated by $\{v_1,\cdots,v_n\}$, then its complexification $V^c$ is generated by elements $w_{jk}:=v_j+iv_k$. Given an endomorphism $\varphi$ of $V$, we can extend it to the complexification, $\varphi^c$, and we can define the unipotent spaces:
$$U:=\{v \in V / \,\exists p, (\varphi-\mbox{Id})^p(v)=0\}$$ $$ U^c:=\{w \in V^c /\, \exists p, (\varphi^c-\mbox{Id})^p(w)=0\}$$
$$w_{jk} \in U^c \Leftrightarrow (\varphi^c-\mbox{Id})^p(w_{jk})=0 \Leftrightarrow (\varphi-\mbox{Id})^p(v_j)+i(\varphi-\mbox{Id})^p(v_k)=0$$
$$\Leftrightarrow \left\{ \begin{array}{l}
(\varphi-\mbox{Id})^p(v_j)=0\\
(\varphi-\mbox{Id})^p(v_k)=0.
\end{array} \right.\Leftrightarrow \left\{ \begin{array}{l}
v_j \in U\\
v_k \in U
\end{array} \right.$$
$$\varphi_s^c(w_{jk})=w_{jk}\Leftrightarrow \varphi_s(v_j)+i\varphi_s(v_k)=v_j+iv_k\Leftrightarrow \left\{ \begin{array}{l}
\varphi_s(v_j)=v_j\\
\varphi_s(v_k)=v_k
\end{array} \right.$$
Then $w \in U^c \Leftrightarrow \varphi_s^c(w)=w$ implies $v \in U \Leftrightarrow\varphi_s(v)=v$.

\smallskip 

\item \textit{$\varphi^c$ has a canonic form:}\\
let $\mbox{ad}_{X_{n+1}}$ be in Jordan form. Then we can consider $\varphi^c$ on $\bigwedge^k\mathbb{C}^n$ for every $k$ to be associated to a matrix made of blocks
$$\left( \begin{array}{ccc}
e^{\lambda t} &  &  *  \\
 & \ddots &   \\
0 &  & e^{\lambda t}
\end{array} \right)$$
Let $\alpha$ be a generator of $\bigwedge^k\mathbb{C}^n$ such that the coefficients of $\varphi^c(\alpha)$ belong to this block, then $\varphi^c(\alpha)=e^{\lambda t}\alpha+\beta$, where $\beta$ is combination of elements belonging to this same block, (the * part).

Now we decompose $\varphi^c$ in the unipotent and semisimple part:

\medskip 

$\varphi^c=\varphi_u\cdot\varphi_s$ where $\varphi_u$ is made of blocks $\left( \begin{array}{ccc}
1 &  &  \star  \\
 & \ddots &   \\
0 &  & 1
\end{array} \right)$ and the semisimple part $\varphi_s$ consists of diagonal blocks of the form $e^{\lambda t}\mbox{Id}$. 

\smallskip

This means that $\varphi_u(\alpha)=\alpha+\beta'$, where $\beta'$ is combination of elements belonging to this same block, (the $\star $ part), $\varphi_s(\alpha)=e^{\lambda t}\alpha$ and $\beta=e^{\lambda t}\beta'$. Then  $\varphi^c(\alpha)=e^{\lambda t}\varphi_u(\alpha)$ and in general $\varphi^c=e^{\lambda t}\varphi_u$ for some $\lambda$.

\smallskip 

\item $\forall p \quad (\varphi^c)^p(\alpha)=e^{p\lambda t}\varphi_u^p(\alpha) :$

we use induction: for $p=2$ we have $$(\varphi^c)^2(\alpha)=\varphi(e^{\lambda t}\varphi_u(\alpha))=e^{\lambda t}\varphi(\varphi_u(\alpha)),$$ but $\beta'$ is combination of elements belonging to the same block, then $$=e^{\lambda t}(e^{\lambda t}\varphi_u(\varphi_u(\alpha)))=e^{2\lambda t}\varphi_u^2(\alpha).$$
If now suppose that the property holds for $p-1$ we can prove it for $p$ in a similar way.

\smallskip

\item $(\varphi^c-\mbox{Id})^k(\alpha)=0 \;\Leftrightarrow\;\varphi_s^c(\alpha)=\alpha:$

let $j$ be the dimension of the block to which $\alpha$ belong, then\\ $(\varphi_u-\mbox{Id})^j(\alpha)=0$. 
\begin{itemize}
\item[\textquotedblleft$\Rightarrow$\textquotedblright:] Let $h\geq \max \{j,k\}$, then $$0=(\varphi^c-\mbox{Id})^h(\alpha)=(e^{\lambda t}\varphi_u-\mbox{Id})^h(\alpha)=$$ $$=[e^{\lambda t}(\varphi_u-\mbox{Id})+(e^{\lambda t}-\mbox{Id})]^h(\alpha)=$$
$$=\sum_{p=0}^{h}\binom{h}{p}(e^{\lambda t}-\mbox{Id})^{h-p}(\alpha)\cdot e^{p\lambda t}(\varphi_u-\mbox{Id})^p(\alpha)=$$
$$=\sum_{p=0}^{h-1}\binom{h}{p}(e^{\lambda t}-\mbox{Id})^{h-p}(\alpha)\cdot e^{p\lambda t}(\varphi_u-\mbox{Id})^p(\alpha)+e^{h\lambda t}(\varphi_u-\mbox{Id})^h(\alpha)$$ but $h\geq j$, then the last summand is $0$ and $$=(e^{\lambda t}-\mbox{Id})(\alpha)\left( \sum_{p=0}^{h-1}\binom{h}{p}(e^{\lambda t}-\mbox{Id})^{h-p-1}(\alpha)\cdot e^{p\lambda t}(\varphi_u-\mbox{Id})^p(\alpha) \right) $$ then $(e^{\lambda t}-\mbox{Id})(\alpha)=0$, i.e. $\varphi_s(\alpha)=e^{\lambda t}\alpha=\alpha$.
\item[\textquotedblleft$\Leftarrow$\textquotedblright:] $\varphi_s(\alpha)=\alpha\; \Leftrightarrow\; e^{\lambda t}=1\;\Leftrightarrow\;\varphi^c(\alpha)=\varphi_u(\alpha)$, then $$(\varphi^c-\mbox{Id})^j(\alpha)=(\varphi_u-\mbox{Id})^j(\alpha)=0.$$
\end{itemize}
\end{enumerate}
\end{proof}

\medskip
This proposition gives also a geometrical meaning to the complexification of $U$, $U^c$: let $V_{\lambda}$ be the subspace of $\mathbb{C}^{n}$ generated by the generators $\alpha$ of $\mathbb{C}^{n}$ such that the coefficients of $\varphi^c(\alpha)$ belong a block of eigenvalue $\lambda$,\\i.e. $\varphi^c(\alpha)=e^{\lambda t}\varphi_u(\alpha)$, then $$U^c=\bigoplus_{\{i_1,\cdots,i_k\}\subseteq \{1,\cdots,n\} ,\,\sum_{p}\lambda_{i_p}t=0}V_{\lambda_{i_1}}\bigwedge\cdots \bigwedge V_{\lambda_{i_k}}$$

\bigskip \bigskip 

Now we prove a property of $U$ that we use in the next sections to study formality of $S$.

\begin{prop}\label{prodotti}
For every $ \alpha, \beta \in H^*(\mathbb{R}^n)$, where $\mathbb{R}^n$ is the n-dimensional abelian Lie algebra, if $\alpha$ and $\beta \in U$ then also  $\alpha\wedge\beta \in U$.
\end{prop}

\begin{proof}
Due to Proposition \ref{kasuya} this proof is very simple: $\alpha$ and $\beta \in U$ is equivalent to $\varphi_s(\alpha)=\alpha$ and $\varphi_s(\beta)=\beta$, then $$\varphi_s(\alpha\wedge\beta)=\varphi_s(\alpha)\wedge\varphi_s(\beta)=\alpha\wedge\beta.$$
\end{proof}

\medskip 

\begin{remark}\label{zero}
$U$ is a submodule of $H^*(\mathbb{R}^n)$, then also in $U$ the zero class is represented only by the zero form in $\bigwedge^*(\mathbb{R}^n)$.
\end{remark}

\bigskip

\section{Formality}\label{formal}

We begin stating two equivalent definitions of $s$-formality and formality, \cite{FM1}, \cite{FM2}, \cite{FOT}:

\begin{definition}\label{def1}
A cdga $(\bigwedge V,d)$ is $s$-formal if there is a cdga homomorphism
$\psi: \bigwedge V^{\leq s}\to H^*(\bigwedge V)$, such that the map $\psi^\ast: H^*(\bigwedge V^{\leq s})\to H^*(\bigwedge V)$ induced on cohomology is equal to the map $i^\ast: H^*(\bigwedge V^{\leq s}) \to H^*(\bigwedge V)$ induced by the inclusion $i: \bigwedge V^{\leq s} \to \bigwedge V$.
\end{definition}

\begin{definition}
A minimal cdga $(\bigwedge V,d)$ is $s$-formal if for every $i\leq s \\ V^i=C^i\oplus N^i$ such that
\begin{itemize}
\item $d(C^i)=0$
\item $d$ is injective on $N^i$
\item $\forall n \in I_s:= \bigwedge V^{\leq s}\cdot N^{\leq s}$ such that $dn=0$, then $n$ is exact in $\bigwedge V$.
\end{itemize}
\end{definition}

We say that $(\bigwedge V,d)$ is \textit{formal} if it is $s$-formal $\forall s\geq 0$, in particular this means

\begin{definition}
A cdga $(\bigwedge V,d)$ is formal if there exists a cgda homomorphism $\psi: \bigwedge V\rightarrow H^*(\bigwedge V)$ that induces the identity in cohomology.
\end{definition}

\begin{definition}\label{def4}
A minimal cdga $(\bigwedge V,d)$ is formal if $ V=C\oplus N$ such that
\begin{itemize}
\item $d(C)=0$
\item $d$ is injective on $N$
\item $\forall n \in I:= \bigwedge V\cdot N$ such that $dn=0$, then $n$ is exact in $\bigwedge V$.
\end{itemize}
\end{definition}

\medskip 

We denote by $(\mathcal{M}_U,d)$ the minimal cdga $(\bigwedge Y, d_Y)$ and by $(\mathcal{M}_S,D)$ the minimal model $(\bigwedge(X\oplus Y),D)$ of $S$.
\\Proposition \ref{prodotti} implies the following

\begin{prop}\label{Uformal}
$(\mathcal{M}_U,d)$ is always formal.
\end{prop}

\begin{proof}
Consider $U$ as a vector space and define $A$ as the subspace of $U$ spanned by generators of $U$ that are wedge of generators of lower degree, and $B$ as the subspace of $U$ spanned by generators of $U$ that are  not wedge of generators of lower degree. Then $U=A\oplus B$.

Using the notation of Definition \ref{def4} we have that if $Y=C\oplus N$,then $C\cong B$ as vector spaces. Then by Proposition \ref{prodotti}, the cohomology of $\mathcal{M}_U$ is given by the elements of $B$. This means that for every $b \in B$ exist $c_b \in C$ such that $[c_b]\cong b$ by the isomorphism $\rho^*$, then $dc_b=0$ and $c_b$ is not exact. Moreover, every $n \in N$ is not closed.

Suppose that there exists a closed element in $\mathcal{M}_U$ which is not a generator and that it lies in $I$. This means that it is a product of two elements and at least one of them is not closed. By Proposition \ref{prodotti} the cohomology of $\mathcal{M}_U$ is given only by the elements of $B$, so this element must be also exact. Otherwise $H^*(\mathcal{M}_U)\ncong U$.

Then by Definition \ref{def4} $(\mathcal{M}_U,d)$ is formal.
\end{proof}

\bigskip 

Now consider the minimal model $(\mathcal{M}_S,D)$ of the solvmanifold $S$. \\By definition $DA=0$ and \begin{equation}\label{D1}\forall x \in Y \;
Dx =\langle \begin{array}{l}
dx \quad \mbox{or}\\
dx +yA \quad \mbox{with}\; y \in \Lambda Y^{< x}
\end{array}.\end{equation}

A generic element in $(\mathcal{M}_S,D)$ has form $s=x+yA$ with $x,y \in \mathcal{M}_U$, then $s$ is closed if and only if $Dx+D(y)A=0$. \\Suppose $Dx=dx+x'A$ and $Dy=dy+y'A$ ($x'$ and $y'$ can be also zero and we will use this notation from now on), then 
\begin{equation}\label{D2} Ds=dx+(x'+dy)A=0 \;\, \mbox{if and only if} \; \left\{ \begin{array}{l}
dx=0\\
 x'+dy=0.
\end{array} \right.\end{equation}
If $s$ is also exact, i.e. there exists $r=p+qA$ with $p$ and $q \in \mathcal{M}_U$ such that $Dr=s$, then $\left\{ \begin{array}{l}
x=dp\\
y=p'+dq
\end{array} \right.$

\medskip

\begin{definition}
A cdga $\mathcal{A}$ is of $k$-finite type if $\forall i\leq k \; \mathcal{A}^i$ is a finite dimensional vector space.
\end{definition}

\begin{remark}
Obviously $\mathcal{M}_S$ is of $k$-finite type if and only if $\mathcal{M}_U$ is of $k$-finite type.
\end{remark}

We can now prove the main result:

\begin{theorem}\label{formality}
If $\mathcal{M}_S$ is of $k$-finite type, then $S$ is $k$-formal if and only if $\ker D_i|_{\mathcal{M}_U}=\ker d_i \; \forall i\leq k$, where with $d_i$ we mean $d|_{\mathcal{M}_U^i}$.
\end{theorem}

\begin{proof}
Suppose that for some $i\leq k \; \ker D_i|_{\mathcal{M}_U}\subsetneq\ker d_i $, then there exists $x \in \mathcal{M}_U^i$ such that $dx=0$, but $Dx\neq 0$. This means for (\ref{D1}) that $Dx=yA$ with $0\neq y \in \mathcal{M}_U^{<x}$, then $D(Ax)=0$ and $x \in N^i$, so $Ax \in I_k$ is closed.

If it is not exact, then $\mathcal{M}_S$ is not $k$-formal, otherwise  there exists an element of degree $i \; x^1\in \mathcal{M}_S^{>x}$ such that $Dx^1=Ax$, then $x^1 \in N^i$ and again $Ax^1 \in I_k$ is closed. If it is not exact $\mathcal{M}_S$ is not $k$-formal, otherwise there exists another element of degree $i \; x^2\in \mathcal{M}_S^{>x^1>x}$ such that $Dx^2=Ax^1$ and so on, but $\mathcal{M}_S$ is of $k$-finite type, then exists $p \in \mathbb{N}$ such that $D(Ax^p)=0$ not exact and so $\mathcal{M}_S$ is not $k$-formal.

\medskip 

Now suppose that $\ker D_i|_{\mathcal{M}_U}=\ker d_i \; \forall i\leq k$. Recall that in Proposition \ref{Uformal} we used only Proposition \ref{prodotti} to prove formality of $\mathcal{M}_U$. Then if we prove an analogous property for $H^{\leq k}(S)$ we can use again Definition \ref{def4} and obtain $k$-formality for $S$.
\\Let $0\neq \alpha=[s_{\alpha}]$ and $0\neq \beta=[s_{\beta}]$ be two elements of $H^{\leq k}(S)$, then $Ds_{\alpha}=Ds_{\beta}=0$ and they are not exact. This means that if $s_{\alpha}=x_{\alpha}+y_{\alpha}A$ and $s_{\beta}=x_{\beta}+y_{\beta}A$, then $dx_{\alpha}=dy_{\alpha}=dx_{\beta}=dy_{\beta}=0$ and do not  exist $r_{\alpha}=p_{\alpha}+q_{\alpha}A$ and $r_{\beta}=p_{\beta}+q_{\beta}A$ such that $Dr_{\alpha}=s_{\alpha}$ and $Dr_{\beta}=s_{\beta}$. This implies that $dp_{\alpha}\neq x_{\alpha}$ and $dp_{\beta}\neq x_{\beta}$, then $[x_{\alpha}]\neq 0$ and $[x_{\beta}]\neq 0$ in $H^{\leq k}(\mathcal{M}_U)\cong U^{\leq k}$.

If we prove that also $s_{\alpha}\cdot s_{\beta}$ is not exact, then $0 \neq \alpha\cdot\beta \in H^*(S)$ and we have $k$-formality. 

Suppose by contradiction that $s_{\alpha}\cdot s_{\beta}=x_{\alpha}\cdot x_{\beta}+(y_{\alpha}\cdot x_{\beta}+x_{\alpha}\cdot y_{\beta})A$ is exact, then there exists $r=p+qA$ such that $Dr=s_{\alpha}\cdot s_{\beta}$, but in particular this implies that $dp=x_{\alpha}\cdot x_{\beta}$, then $[x_{\alpha}\cdot x_{\beta}]=0$ in $H^*(\mathcal{M}_U)\cong U$ that is impossible by Proposition \ref{prodotti}, then also $s_{\alpha}\cdot s_{\beta}$ is not exact.
\end{proof}

The Main Theorem is obviously a direct consequence of this theorem.

\bigskip

\section{Symplectic structures}\label{symplectic}

Suppose that $S=\mathbb{R}\ltimes\mathbb{R}^{2n-1}$ has dimension $2n$. Recall that a symplectic form on $S$ is $\omega \in \bigwedge^2 S$ such that $d\omega=0$ and $\omega^n \neq 0$. We denote with $\{\alpha^1, \cdots , \alpha^{2n-1} \}$ the basis of $\bigwedge^1 \mathbb{R}^{2n-1}$ and with $\{ \alpha^{2n} \}$ the basis of $\bigwedge^1 \mathbb{R}$

\smallskip 

\begin{definition}
If $M$ is a $(2n-1)$-dimensional manifold a co-symplectic structure on $M$ is a couple $(F, \eta)$ where $F$ is a $2$-form, $\eta$ is a $1$-form on $M$, both are closed and $F^{n-1}\wedge \eta \neq 0$.
\end{definition}

In particular we call a \textit{co-symplectic structure on $U$} a co-symplectic structure $(F, \eta)$ on $\mathbb{R}^{2n-1}$ such that $[F],[\eta] \in U$. Observe that every form on $\mathbb{R}^{2n-1}$ is closed, so the only necessary condition to get this structure is the non-degeneracy.

\medskip 

Let $(F, \eta)$ be a co-symplectic structure on $U$. This means that $$F:=\sum_{1\leq i < j \leq 2n-1} a_{ij}\alpha^{ij}, \quad \eta:=\sum_{1\leq k \leq 2n-1} b_k\alpha^{k}$$
$[F],[\eta] \in U$ and $F^{n-1}\wedge \eta \neq 0$.

\medskip 

Now consider the minimal model $\mathcal{M}_S$ of $S$. If $A$ is the generator we add to $U$ from $\bigwedge^* \mathbb{R}$, then with the notation of Theorem \ref{fibration} we have $\sigma(A)=\alpha^{2n}$ and then also 
$$\begin{array}{rcl}
\tau: \mathcal{M}_S & \rightarrow & \Lambda^* S \\
A & \mapsto &\alpha^{2n} \\
\mathcal{M}_U & \mapsto & \rho(\mathcal{M}_U)\subset \Lambda^* \mathbb{R}^{2n-1}
\end{array}$$

$[F],[\eta] \in U$ then there exist $x \in \mathcal{M}_U^2$ and $y \in \mathcal{M}_U^1$ such that $\rho^*([x])=[F]$ and $\rho^*([y])=[\eta]$. \\But in $U \subset H^*(\mathbb{R}^{2n-1})$ we do not have exact forms. So $\rho(x)=F$ and $\rho(y)=\eta$. 

Therefore $dx=dy=0$ and if $s:=x+yA \in \mathcal{M}_S^2$, $Ds=Dx=x'A$.
$$s^n=(x+yA)^n=\sum_{p=0}^{n}\binom{n}{p}x^{n-p}y^pA^p=x^n+nx^{n-1}yA$$ because both $y$ and $A$ have odd degree and then their powers are $0$. But
$$\rho(x^{n-1}y)=(\rho(x))^{n-1}\rho(y)=F^{n-1}\wedge \eta \neq 0,$$ then $x^{n-1}y \neq 0$ in $\mathcal{M}_U$ and so $x^{n-1}yA \neq 0$ in $\mathcal{M}_S$.

\smallskip 

$x^n \in \mathcal{M}_U$, then $x^n \neq -nx^{n-1}yA \in \mathcal{M}_S$, then $s^n\neq 0$ in $\mathcal{M}_S$.

In particular $\omega:=\tau(s)=\tau(x)+\tau(y)\tau(A)=F+\eta\wedge \alpha^{2n}$ is a $2$-form on $S$ and $$\omega^n=\tau(s^n)=(\tau(x))^n+n(\tau(x))^{n-1}\tau(y)\tau(A)=F^n+nF^{n-1}\wedge\eta\wedge\alpha^{2n}.$$

$F^n=0$ because it is in $\bigwedge(\alpha^1, \cdots,\alpha^{2n-1})$ and $F^{n-1}\wedge \eta \neq 0$ by hypothesis, then also $\omega^n=nF^{n-1}\wedge\eta\wedge\alpha^{2n} \neq 0$.
\\Since $d\omega=\tau(Ds)$ by Definition \ref{morp}, if $x'=0, \; \omega$ is closed and we have a symplectic structure on $S$.

\medskip 

We have then proved the following proposition:

\begin{prop}
If $D_2|_{\mathcal{M}_U}=d_2$ and there exists a co-symplectic structure on $U$, then there exists a symplectic structure on $S$.
\end{prop} 

\bigskip

\section{Examples}

We conclude giving two examples of computation:

\subsection{An example in dimension 6}
Consider the almost abelian solvmani-\\fold $S_6$ defined by the action of
$$\mbox{ad}_{X_6}=\left( \begin{array}{ccccc}
0 & 1 & 0 & 0 & 0 \\
0 & 0 & 1 & 0 & 0 \\
0 & 0 & 0 & 0 & 0 \\
0 & 0 & 0 & 0 & 1 \\
0 & 0 & 0 & -1 & 0
\end{array} \right)$$
with lattice generated by $t=2\pi$.

The Lie algebra associated to this solvmanifold  in \cite{bock} is called $\mathfrak{g}_{6.10}^{a=0}$.

According to the method developed in \cite{CF2} and \cite{Guan}, this solvmanifold is diffeomorphic to the $6$-dimensional, almost abelian, completely solvable solvmanifold $\tilde{G}/\Gamma_{2\pi}$ with $\tilde{G}=\mathbb{R}\ltimes_{\tilde{\varphi}}\mathbb{R}^5$ and 
$$\tilde{\varphi}=\left(\begin{array}{ccccc}
0 & 1 & 0 & 0 & 0 \\
0 & 0 & 1 & 0 & 0 \\
0 & 0 & 0 & 0 & 0 \\
0 & 0 & 0 & 0 & 0 \\
0 & 0 & 0 & 0 & 0
\end{array}\right).$$

Then its cohomology groups are isomorphic to those of the Lie algebra $\tilde{\mathfrak{g}}$ given by 
$[X_2,X_6]=X_1, \; [X_3,X_6]=X_2$, \cite{noi}.

In particular we have 

$$\begin{array}{rl}
H^1(S_6)=&\langle \alpha_3,\,\alpha_4,\,\alpha_5,\,\alpha_6\rangle \\ 
H^2(S_6)=&\langle \alpha_{16},\,\alpha_{23},\,\alpha_{34},\,\alpha_{35},\,\alpha_{45},\,\alpha_{46},\,\alpha_{56}\rangle \\
H^3(S_6)=&\langle \alpha_{123},\,\alpha_{126},\,\alpha_{146},\,\alpha_{156},\,\alpha_{234},\,\alpha_{235},\,\alpha_{345},\,\alpha_{456}\rangle
\end{array}$$

Now we compute $U$: $\varphi=e^{2\pi\mbox{ad}_{X_6}}$, then 
$$\begin{array}{rl}
\varphi(\alpha^1)=& \alpha^1+2\pi\alpha^2+2\pi^2\alpha^3, \\ 
\varphi(\alpha^2)=& \alpha^2+2\pi\alpha^3, \\ 
\varphi(\alpha^3)=& \alpha^3, \\
\varphi(\alpha^4)=& \alpha^4, \\ 
\varphi(\alpha^5)=& \alpha^5.\end{array}$$

In this case $\forall \, i=1,\cdots,5\quad \alpha^i \in U$, then $U\equiv H^*(\mathbb{R}^5)$ and $$\mathcal{M}_U\equiv\mathcal{M}_U^1=(\bigwedge(e,f,z,p,q),0).$$
\\Knowing the cohomology groups of the solvmanifold we can compute its minimal model: {\small $$\mathcal{M}_S=(\bigwedge(A,e,f,z,p,q),D),\;DA=De=Df=Dz=0,\,Dp=eA,\,Dq=pA$$}
with the map $\tau:\mathcal{M}_S \rightarrow \bigwedge^*S$ given by  $$\tau(A)=\alpha^6,\tau(e)=\alpha^3,\tau(f)=\alpha^4,\tau(z)=\alpha^5,\tau(p)=\alpha^2,\tau(q)=\alpha^1.$$

Then for Theorem \ref{formality} $S_6$ is not 1-formal.

\medskip 

Now consider the symplectic forms on $S_6$. \\In this case the generic co-symplectic structure on $U$ is given by $$F=\sum_{1\leq i<j\leq 5}a_{ij}\alpha^{ij} \quad \mbox{and}\quad \eta=\sum_{1\leq k\leq 5}b_k\alpha^k$$
$$\begin{array}{rcl}
\mbox{with}\;F^2\wedge\eta \neq 0 &\Leftrightarrow& b_5(a_{12}a_{34}-a_{13}a_{24}+a_{14}a_{23})+\\
&&+b_4(a_{12}a_{35}-a_{13}a_{25}+a_{15}a_{23})+\\
&&+b_3(a_{12}a_{45}-a_{14}a_{25}+a_{15}a_{24})+\\
&&+b_2(a_{13}a_{45}-a_{14}a_{35}+a_{15}a_{34})+\\ 
&&+b_1(a_{23}a_{45}-a_{24}a_{35}+a_{25}a_{34})\neq 0.\end{array}$$

Let $x \in \mathcal{M}_U^2$ and $y \in \mathcal{M}_U^1$ such that $\tau(x)=F$ and $\tau(y)=\eta$, then 
$$x=a_{12}qp+a_{13}qe+a_{14}qf+a_{15}qz+a_{23}pe+a_{24}pf+a_{25}pz+a_{34}ef+a_{35}ez+a_{45}fz$$
and $y=b_1q+b_2p+b_3e+b_4f+b_5z$.

\smallskip 

The element $s:=x+yA \in \mathcal{M}_S^2$ is closed if and only if $$x'=a_{12}eq+a_{13}ep+a_{14}fp+a_{15}zp-a_{24}ef-a_{25}ez=0$$ that is if and only if $ a_{12}=a_{13}=a_{14}=a_{15}=a_{24}=a_{25}=0$.

Then if we consider $F=a_{23}\alpha^{23}+a_{34}\alpha^{34}+a_{35}\alpha^{35}+a_{45}\alpha^{45}$ and\\ $\eta=\sum_{1\leq k\leq 5}b_k\alpha^k$ with $b_1\neq 0,\,a_{23}\neq 0,\,a_{45}\neq 0$, we have a symplectic structure on $S$ given by $\omega:=F+\eta\wedge\alpha^6$.

\smallskip

\subsection{An example in dimension 8}
Now consider the almost abelian solvmanifold $S_8$ defined by the action of 
$$\mbox{ad}_{X_8}=\left(\begin{array}{ccccccc}
0&1&0&0&0&0&0\\
0&0&1&0&0&0&0\\
0&0&0&0&0&0&0\\
0&0&0&b&1&0&0\\
0&0&0&-1&b&0&0\\
0&0&0&0&0&-b&1\\
0&0&0&0&0&-1&-b
\end{array}\right)\qquad b\neq 0,  \; e^{2\pi b}+e^{-2\pi b} \in \mathbb{Z}$$
with lattice generated by $t=2\pi$.

Again this solvmanifold is diffeomorphic to the $8$-dimensional, almost abelian, completely solvable solvmanifold $\tilde{G}/\Gamma_{2\pi}$ with $\tilde{G}=\mathbb{R}\ltimes_{\tilde{\varphi}}\mathbb{R}^7$ and 
$$\tilde{\varphi}=\left(\begin{array}{ccccccc}
0&1&0&0&0&0&0\\
0&0&1&0&0&0&0\\
0&0&0&0&0&0&0\\
0&0&0&b&0&0&0\\
0&0&0&0&b&0&0\\
0&0&0&0&0&-b&0\\
0&0&0&0&0&0&-b
\end{array}\right).$$

Then its cohomology groups are isomorphic to those of the Lie algebra $\tilde{\mathfrak{g}}$ given by 
$\quad  [X_2,X_8]=X_1, [X_3,X_8]=X_2, [X_4,X_8]=bX_4, [X_5,X_8]=bX_5,$ \\ $ [X_6,X_8]=-bX_6, [X_7,X_8]=-bX_7.$

In particular we have $H^1(S)=\langle \alpha^3, \alpha^8\rangle$.

\medskip 

Now we compute $U$:
$$\begin{array}{rl}
\varphi(\alpha^1)=& \alpha^1+2\pi\alpha^2+2\pi^2\alpha^3,\\
\varphi(\alpha^2)=& \alpha^2+2\pi\alpha^3,\\
\varphi(\alpha^3)=& \alpha^3,\\
\varphi(\alpha^4)=& e^{2\pi b}\alpha^4,\\
\varphi(\alpha^5)=& e^{2\pi b}\alpha^5,\\
\varphi(\alpha^6)=& e^{-2\pi b}\alpha^6,\\
\varphi(\alpha^7)=& e^{-2\pi b}\alpha^7,\\
\end{array}$$
then $$\begin{array}{rl}
U^1=& \langle \alpha^1, \alpha^2, \alpha^3\rangle, \\
U^2=&\langle \alpha^{12},\alpha^{13},\alpha^{23}, \alpha^{46},\alpha^{47}, \alpha^{56},\alpha^{57}\rangle\\
U^3=&\langle \alpha^{123},\alpha^{146},\alpha^{147}, \alpha^{156},\alpha^{157},\alpha^{246},\alpha^{247}, \alpha^{256},\alpha^{257},\\
&\alpha^{346},\alpha^{347}, \alpha^{356},\alpha^{357} \rangle\\
U^4=&\langle \alpha^{1246},\alpha^{1247},\alpha^{1256},\alpha^{1257},\alpha^{1346},\alpha^{1347},\alpha^{1356},\alpha^{1357},\\
&\alpha^{2346},\alpha^{2347},\alpha^{2356},\alpha^{2357},\alpha^{4567} \rangle\\
U^5=&\langle \alpha^{12346},\alpha^{12347},\alpha^{12356},\alpha^{12357},\alpha^{14567},\alpha^{24567},\alpha^{34567} \rangle\\
U^6=&\langle \alpha^{124567},\alpha^{134567},\alpha^{234567} \rangle \\
U^7=&\langle \alpha^{1234567} \rangle.\end{array}$$

\medskip

The minimal cdga $\mathcal{M}_U$ is quite difficult to compute, indeed the big dimension of $U^2$ implies a need of many generators in degree $2$, and then many relations to check to get the cohomology isomorphism.\\
Fortunately we do not need to construct all $\mathcal{M}_U$ and $\mathcal{M}_S$ to understand if the solvmanifold is formal: 
we can simply find out that $\mathcal{M}_U^1=(\bigwedge(x,y,z),0)$, but $H^1(S)=\langle \alpha^3, \alpha^8\rangle$, then $$\mathcal{M}_S^1=(\bigwedge(A,x,y,z),D)\;\mbox{with}\; DA=Dx=0, Dy=xA,\,Dz=yA$$ and so for Theorem \ref{formality} $S$ is not $1$-formal.

\bigskip 

\noindent \textbf{Acknowledgments.}  We would like to thank Hisashi Kasuya, Andrea Mori and Luis Ugarte for many useful comments and suggestions.

\end{document}